\documentclass[11pt]{article}
\usepackage[english]{babel}
\usepackage{epsfig}
\usepackage{amsmath}
\usepackage{amsthm}
\usepackage{amssymb}

\newtheorem{theorem}{Theorem}[section]
\newtheorem{definition}[theorem]{Definition}

\newtheorem{lemma}[theorem]{Lemma}

\newtheorem{remark}[theorem]{Remark}

\title{\bf Conformal mappings revisited in the octonions and Clifford algebras of arbitrary dimension}

\author{
Rolf S\"{o}ren Krau{\ss}har\\
Fachbereich Mathematik\\
Erziehungswissenschaftliche Fakult\"at\\
Universit\"at Erfurt\\
Nordh\"auser Str. 63\\
99089 Erfurt, Germany\\
soeren.krausshar@uni-erfurt.de }

\begin{document}
\maketitle
\begin{abstract}  
	In this paper we revisit the concept of conformality in the sense of Gauss in the context of octonions and Clifford algebras. We extend a characterization of conformality in terms of a system of partial differential equations and differential forms using special orthonormal sets of continuous functions that have been used before in the particular quaternionic setting. The aim is to describe to which higher dimensional algebras this characterization can exactly be extended and under which circumstances. It turns out to be crucial that this characterization requires a domain of definition that lies in a subalgebra that has the norm composition property and that is either associative (Clifford algebra case) or at least alternative (octonionic case). The orthonormal frames are elements of the spin group Spin$(n+1)$. We round off by relating the nature of the orthonormal frames to the associated M\"obius transformation which are related to $SO(9,1)$ in the octonionic case and to the Ahlfors-Vahlen group in the case of a Clifford algebra.    
\end{abstract}
{\bf Keywords}: conformal maps, octonions, Clifford algebras, M\"obius transformations, variable structural sets, spin groups \\[0.1cm] 
\noindent {\bf Mathematical Review Classification numbers}: 30G35\\

\section{Introduction}

In the two-dimensional setting the set of angle, length and orientation preserving maps over a domain (called conformal maps in the strict sense) coincides with all holomorphic functions that are injective over that domain and can hence be characterized by the Cauchy-Riemann differential equation $\frac{\partial f}{\partial y} = -i \frac{\partial f}{\partial x}$.  If one drops the property of preserving orientation, when we are talking about conformal maps in the sense of Gauss, then one additionally gets the solutions to the conjugated Cauchy-Riemann equation $\frac{\partial f}{\partial y} = i \frac{\partial f}{\partial x}$.

As it is extremely well-known, Liouville's theorem tells us that in Euclidean spaces of real dimension $n\ge 3$ the set of conformal maps is restricted to the set of M\"obius transformations which does not fit within the framework of null-solutions to the generalized Cauchy-Riemann equation $\frac{\partial f}{\partial x_0} + \sum\limits_{j=1}^n e_j \frac{\partial f}{\partial x_j}$ neither in the context of Clifford algebras nor in that of the Cayley-Dickson algebras that are generated by the usual doubling process of doubling the complex numbers in the classical way. 

However, as shown in \cite{BSN,DKM,KraMal2001} for the particular quaternionic case, conformal maps, and hence equivalently M\"obius transformations related to $GL(2,\mathbb{H})$, can be characterized by a different equation that generalizes the Cauchy-Riemann equation in another and different way.  While in the complex plane there only exist two imaginary numbers of lenght $1$, namely $i$ and $-i$, in the quaternions there exists a whole sphere of elements with length $1$ and vanishing real part. Conformal maps in the sense of Gauss turned out to be locally characterizable by the following generalized version of the Cauchy-Riemann system of the form $\frac{\partial f}{\partial x_j} = \Psi_j(z) \frac{\partial f}{\partial x_0}$, $j=1,2,3$ where the elements $\Psi_j(z)$ form an orthonormal system of unit quaternions with vanishing real part. Such an arbitrary orthonormal system is often called a structural set, cf. \cite{SV}.  Furthermore, the elements $\Psi_i(z)$ can locally vary from point to point. In general, they are not constant but continuous functions. In the complex case, this system reduces to the classical Cauchy-Riemann or conjugated Cauchy-Riemann system, because the constants $i$ and $-i$ are the only admissible choices for such functions in $\mathbb{C}$. 

The aim of this paper is to explain to which higher dimensional algebras (beyond the quaternions) this characterization can be carried over and which algebraic conditions are required to achieve this.  

The Cayley-Dickson doubling process generates the quaternions from doubling the complex numbers which in turn arise from doubling the reals. If one continues applying the usual Cayley-Dickson process to the quaternions, then one obtains the non-associative octonions, which nevertheless still form a normed alternative division algebra. In the next step of this doubling construction, the property of being a composition and alternative algebra is lost. Furthermore, one gets zero divisors.  

Another way to extend the quaternions is offered by Clifford algebras which have the advantage of being associative in any dimension. However, one has to deal with zero divisors, as soon as being beyond quaternions. Additionally one loses the multiplicativity of the norm. However, a zero-divisor free subspace where the composition property of the norm still holds is the space of paravectors which is  isomorphic to $\mathbb{R}^{n+1}$. The paravectors generate the multiplicative Clifford group. In turn, the Clifford group contains as special subgroups the famous pin and spin group which are double covers of the $O(n+1)$ and $SO(n+1)$, respectively.      

To establish a generalization of the particular characterization  of conformality as indicated above, we take domains of definition that lie in a subalgebra that has the norm composition property and that is either associative (Clifford algebra case) or at least alternative (octonionic case). 

We show that one can extend this characterization both to the octonionic setting and to the setting of real paravector-valued functions. The structural sets that arise in this generalized context are elements of the spin group Spin$(n+1)$. We round off by relating the nature of the orthonormal frames to the associated M\"obius transformation which are related to $SO(9,1)$ in the octonionic case and to the Ahlfors-Vahlen group $SV(n+1)$ (cf. \cite{EGM87}) in the case of taking a domain that is embedded in a Clifford algebra.  Like in the quaternionic case, the orthonormal frames are only constant if and only if the associated M\"obius transformation is exclusively generated by translations, dilatations and rotations. Whenever a reflection at a sphere is involved, one deals with non-constant but continuous orthonormal structural frames.  

It turns out that this characterization of conformality cannot be carried over to higher dimensional Cayley-Dickson algebras beyond octonions, or to Clifford algebra-valued functions that take other values than paravectors since the argumentation explicitly requires the composition property and the associative or alternative property. 

This paper underlines once more the particularity of the octonionic case in the context of Cayley-Dickson algebras concerning symmetry theories. It currently seems that the octonions provide the most adequate algebra to study the arising symmetries of recent physical attempts of setting up unifying theories in connection with the standard model of particle physics and supergravity, see \cite{Burdik,G,NSM}.     
  
 \section{Clifford algebras vs. octonions}
 There are many possibilities to extend the complex number field to higher dimensions. 
 
 One method is the doubling principle of Cayley-Dickson described in many papers, see for example \cite{Baez,WarrenDSmith}. To start, take two pairs of complex numbers $(a,b)$ and $(c,d)$ and introduce an addition and multiplication operation by  $$
 (a,b)+(c,d) :=(a+c,b+d),\quad\quad (a,b)\cdot (c,d) := (ac-d\overline{b},\overline{a}d+cb). 
 $$ 
 where $\overline{\cdot}$ simply denotes the conjugation (anti-)automorphism  which will be extended by $\overline{(a,b)}:=(\overline{a},-b)$ to the set of pairs $(a,b)$. 
 We now have obtained the Hamiltonian skew field of quaternions. Each quaternion can be written in the form $z=x_0 + x_ 1e_1 + x_2 e_2 + x_3 e_3$ where $e_i^2=-1$ for $i=1,2,3$ and $e_1 e_2 = e_3$, $e_2 e_3 = e_1$, $e_3 e_1 = e_2$ and $e_i e_j = - e_j e_i$ for all mutually  distinct $i,j$ from $\{1,2,3\}$. Obviously, the commutativity has been lost in this doubling process, but $\mathbb{H}$ is still associative.
   
 Applying the next doubling by considering pairs of quaternions we arrive at the octonions $\mathbb{O}$ which are often called Cayley numbers, too. Octonions have the form 
 $$
 z = x_0 + x_1 e_1 + x_2 e_2 + x_3 e_3 + x_4 e_4 + x_5 e_5 + x_6 e_6 + x_7 e_7
 $$
 where $e_4=e_1 e_2$, $e_5=e_1 e_3$, $e_6= e_2 e_3$ and $e_7 = e_4 e_3 = (e_1 e_2) e_3$. 
 Like in the quaternionic case, we have $e_i^2=-1$ for all $i =1,\ldots,7$ and $e_i e_j = -e_j e_i$ for all mutual distinct $i,j \in \{1,\ldots,7\}$. The multiplication is explained by the table
\begin{center}
 \begin{tabular}{|l|rrrrrrr|}
 $\cdot$ & $e_1$&  $e_2$ & $e_3$ & $e_4$ & $e_5$ & $e_6$  & $e_7$ \\ \hline
 $e_1$  &  $-1$ &  $e_4$ & $e_5$ & $-e_2$ &$-e_3$ & $-e_7$ & $e_6$ \\
 $e_2$ &  $-e_4$&   $-1$ & $e_6$ & $e_1$ & $e_7$ & $-e_3$ & $-e_5$ \\
 $e_3$ &  $-e_5$& $-e_6$ & $-1$  & $-e_7$&$e_1$  & $e_2$  & $e_4$ \\
 $e_4$ &  $e_2$ & $-e_1$ & $e_7$ & $-1$  &$-e_6$ & $e_5$  & $-e_3$\\
 $e_5$ &  $e_3$ & $-e_7$ & $-e_1$&  $e_6$&  $-1$ & $-e_4$ & $e_2$ \\
 $e_6$ &  $e_7$ &  $e_3$ & $-e_2$& $-e_5$& $e_4$ & $-1$   & $-e_1$ \\
 $e_7$ & $-e_6$ &  $e_5$ & $-e_4$& $e_3$ & $-e_2$& $e_1$  & $-1$ \\ \hline 	
 \end{tabular}
\end{center}
The octonions are not associative anymore. However, they still form an alternative and composition algebra.  
Although one has no associativity anymore one still has the weaker Moufang rule $(ab)(ca) = a((bc)a)$ for all $a,b,c \in \mathbb{O}$. Putting $c=1$ yields the flexibility condition $(ab)a= a(ba)$. Moreover,  
$$
(a\overline{b})b = \overline{b}(ba) =a(\overline{b}b)=a(b \overline{b})
$$  
for all $a,b \in \mathbb{O}$. 
Next following \cite{CDieckmann} Proposition 1.6, one also has the important rule
$\Re\{b(\overline{a}a)c\} =\Re\{(b \overline{a})(ac)\}$ for all $a,b,c \in \mathbb{O}$ which we will exploit at some crucial points to prove the main results of this paper. 

$\mathbb{R},\mathbb{C},\mathbb{H}$ and $\mathbb{O}$ are the only real normed division algebras. Let $a = a_0 + \sum\limits_{i=1}^7 a_i e_i$ be an octonion represented with the seven imaginary units as mentioned above. We call $a_0$ the real part of $a$ and write $\Re{a} = a_0$. The conjugation leaves the real part invariant, but $\overline{e_j}=-e_j$ for all $j =1,\ldots,7$. On two general octonions $a,b \in \mathbb{O}$ one has $\overline{a\cdot b} = \overline{b}\cdot \overline{a}$. 

The Euclidean norm and the Euclidean scalar product from $\mathbb{R}^8$ naturally extends to the octonionic case by $\langle a,b \rangle := \sum\limits_{i=0}^7 a_i b_i = \Re\{a \overline{b}\}$ and $|a|:= \sqrt{\langle a,a\rangle} = \sqrt{\sum\limits_{i=0}^7 a_i^2}$. We have the important norm composition property $|a\cdot b|=|a|\cdot|b|$ for all $a,b \in \mathbb{O}$. Every non-zero element $a \in \mathbb{O}$ is invertible with $a^{-1} =\overline{a}/|a|^2$. 

Unfortunately, if one continues applying the Cayley-Dickson doubling process, then one even loses the properties of a composition and an alternative algebra and one gets zero-divisors. 
\par\medskip\par
A parallel world alternatively to Cayley-Dickson algebras of higher dimensional generalizations of complex numbers is offered by the construction of Clifford algebras. In contrast to Cayley-Dickson algebras, {\em all} Clifford algebras remain associative. Let us follow for instance \cite{bds}. Suppose that $\{e_1,e_2,\ldots,e_n\}$ is the standard orthonormal basis of the Euclidean vector space $\mathbb{R}^n$. The associated real Clifford algebra $Cl_n$ is the free algebra generated by $\mathbb{R}^{n}$ modulo the relation $a^2 = -|a|^2 e_0$ where $e_0=1$ is its multiplicative neutral element. The relation implies the multiplication rules $e_i^2=-1$ for all $i=1,\ldots,n$ and $e_i e_j = -e_j e_i$ for all $i \neq j$. 

An arbitrary element $a$ of the Clifford algebra $Cl_n$ can be written in the form $a = \sum_{A \subseteq \{1,\ldots,n\}} a_A e_A$ where the basis elements $e_A$ are products of the elements $e_i$, i.e. $e_A := e_{l_1} e_{l_2} \ldots e_{l_n}$ where 
$1 \le l_1 < \cdots < l_r \le n$, $e_{\emptyset} = e_0 = 1$. So, 
$$
a = a_0 + a_1 e_1 + \cdots + a_n e_n + a_{1,2} e_1 e_2 + \cdots  + a_{n-1,n} e_{n-1} e_n + \cdots + a_{1,2,\ldots,n} e_1 e_2 \cdots e_n.
$$ 
The conjugation anti-automorphism can be defined on the full Clifford algebra by $e_j=-e_j$ for each $j=1,\ldots,n$ and $\overline{a b} = \overline{b} \; \overline{a}$ for each $a,b \in Cl_n$.  We will also use the reversion anti-autmorphism which is defined by $e_j^*=e_j$ for each $j=1,\ldots,n$ and $(a b)^* = b^*  a^*$ for any $a,b \in Cl_n$
We further denote by Sc$(a):=a_0$ the scalar part of $a$ and by $Vec(a) :=\sum_i^n a_i e_i$ the vector part of $a$. Subsequentially, $a_{1,2} e_{1} e_2 + \cdots + a_{n-1,n} e_{n-1} e_n$ is the bivector part and so on. The space $\mathbb{R} \oplus \mathbb{R}^{n}$ is called the space of paravectors and will be identified with $\mathbb{R}^{n+1}$ in what follows. The space of all elements of the Clifford algebra that only have a scalar part, a vector part and a bi-vector part is denoted by $\mathbb{R} \oplus \mathbb{R}^{n} \oplus \mathbb{R}^{n,n}$ and is also important in the sequel of this paper. 

The first Clifford algebras $Cl_1$ and $Cl_2$ are isomorphic to $\mathbb{C}$ and $\mathbb{H}$, respectively. Unfortunately, up from $Cl_n$ with $n \ge 3$ one loses the property of a composition algebra and one gets zero-divisors. However, the subspace of paravectors is still normed and zero-divisor free. The Euclidean norm naturally extends to $\mathbb{R} \oplus \mathbb{R}^n$ by $|a| = \sqrt{\sum_{i=0}^n a_i^2}$ and every non-zero paravector is invertible with $a^{-1} = \overline{a}/|a|^2$. On the whole Clifford algebra one can introduce a scalar product of the form $\langle a,b\rangle := {\rm Sc}(a \overline{b})$ and a (pseudo-)norm $|a| := \sqrt{{\rm Sc}(a \overline{a})}$. However, on the general Clifford algebra, one only has an inequality of the form $|ab| \le 2^{n/2} |a||b|$ for $a,b \in Cl_n$.    

Nevertheless, the composition property $|ab| = |a||b|$ remains true for all paravectors $a,b \in \mathbb{R}^{n+1}$ and even for products of paravectors. This will also be crucially applied in the sequel of this paper.

\section{Main results}

This section provides us with a generalization some of the results presented for the quaternionic case in \cite{BSN,DKM,KraMal2001} to the non-associative octonionic case on the one hand and to arbitrary dimensional paravector spaces in associative Clifford algebras on the other hand. 

As already pointed out in the introduction, this topic in the octonionic setting attracts in the recent time a lot of interest from physicists in the context of extensions of the standard model of particle physics and super gravity, see for instance \cite{Burdik,G,NSM}.  

It should also be clearly mentioned that there are substantial differences between the related octonionic function theories and function theories in Clifford algebras which is also a topic of recent interest, see for instance \cite{KO2018,KO2019,Kra2019}. As classical references we also recommend \cite{GTBook,Nono,XL2000}. 

Following classical textbooks like \cite{BG} and others, as well as our previous work  from the quaternionic setting, we can address conformality in the octonionic setting in the following way:
\begin{definition}
Let $D \subseteq \mathbb{O}$ be a domain. Suppose that $f:D\to \mathbb{O}$ is at least a $C^1(D)$-function, that means that each real component function is supposed to be real differentiable over $D$. Then $f$ is called conformal (in the sense of Gauss) if there is a real positive-valued scaling function $\lambda:D \subseteq \mathbb{O} \to \mathbb{R}^{>0}$ with $|df|^2 = \lambda(z) |dz|^2$, 
where $|\cdot|$ is the usual Euclidean norm from $\mathbb{R}^8$ carried over to the octonions $\mathbb{O}$.  
\end{definition}
The differentials $dz$ and $df$ are defined as usual. $dz$ is an octonionic $1$-form: $$dz = dx_0 + e_1 dx_1 + \cdots + e_7 dx_7$$ and $df = \sum\limits_{i=0}^7\frac{\partial f}{\partial x_i} dx_i$. Note that the partial derivatives $\frac{\partial f}{\partial x_i}$ are octonion-valued. 

Analogously, one defines in the framework of paravector-valued functions embedded in an associative Clifford algebra:
\begin{definition}
Let $D \subseteq \mathbb{R}^{n+1} \subset Cl_n$ be a domain. A paravector-valued $C^1(D)$-function $f:D \to \mathbb{R}^{n+1}$ is called conformal (in the sense of  Gauss) if there is a real positive-valued scaling function $\lambda:D \subseteq \mathbb{R}^{n+1} \to \mathbb{R}^{>0}$ with $
|df|^2 = \lambda(z) |dz|^2,$ 
where $|\cdot|$ is the usual Euclidean norm from the paravector space $\mathbb{R}^{n+1} :=\mathbb{R} \oplus \mathbb{R}^{n}$.  
\end{definition} 
The differentials are defined in the same way as above, compare also with  \cite{Mal1991}. 

One can show in both cases:
\begin{lemma}
Let $D \subseteq \mathbb{O}$ or $D \subseteq \mathbb{R}^{n+1}$, respectively, be a domain. A $C^1(D)$-function $f$ is conformal in the sense of Gauss if and only if it satisfies at each point $z \in D$:
\begin{eqnarray}
\Big| \frac{\partial f}{\partial x_0}\Big| &=& \Big| \frac{\partial f}{\partial x_i}\Big|,\quad \forall i=1,\ldots,7,\;({\rm resp.}\;\; \forall i=1,\ldots,n)\\
\Re\Big(\frac{\partial f}{\partial x_i} \frac{\partial \overline{f}}{\partial x_j}\Big) &=& 0,  \forall i \neq j,  
\;\;i,j \in \{0,1,\ldots,7\}\\
\Bigg({\rm Sc}\Big(\frac{\partial f}{\partial x_i} \frac{\partial \overline{f}}{\partial x_j}\Big) &=& 0\;(\rm{resp.}\;\;i,j \in \{0,1,\ldots,n\}).\Bigg) \nonumber 
\end{eqnarray}
\end{lemma}
\begin{proof}
	Both for the octonionic case (addressed for $n=7$) and for the paravector-valued cases (addressed for general $n$), we obtain by a direct computation that 
\begin{eqnarray*}
|df|^2 &=& \Big(\sum\limits_{i=0}^n \frac{\partial f}{\partial x_i} dx_i\Big)\cdot \Big(\sum\limits_{i=0}^n \frac{\partial \overline{f}}{\partial x_i} dx_i\Big)\\
& = & \sum\limits_{i=0}^n \Big(\sum\limits_{j=0}^{n} (\frac{\partial f_j}{\partial x_i} )^2  \Big)dx_i^2	+ 
2 \sum\limits_{j<i} \Big(\sum\limits_{l=0}^n \frac{\partial f_l}{\partial x_j}\frac{\partial f_l}{\partial x_i}\Big)dx_j dx_i.
\end{eqnarray*}
The latter expression equals $\lambda(z)|dz|^2 = \lambda(z) \sum\limits_{j=0}^n dx_j^2$ if and only if 
\begin{eqnarray*}
\sum\limits_{j=0}^n (\frac{\partial f_j}{\partial x_i})^2 &=&\lambda(z)\\
\sum\limits_{j=0}^n \frac{\partial f_j}{\partial x_i} \frac{\partial f_j}{\partial x_k} &=& 0,\quad \forall i < k.	
\end{eqnarray*}
\end{proof}
\begin{remark}
Equivalently, one can characterize conformality in the way
$$
\langle \frac{\partial f}{\partial x_i},\frac{\partial f}{\partial x_k}\rangle = \delta_{ik} \lambda(z),
$$		
where $\delta_{ik}$ represents the usual Kronecker symbol.
\end{remark}
As a generalization of the quaternionic case, we can also establish a characterization of conformality in the octonionic case in the following way:
\begin{theorem}
Let $D \subseteq \mathbb{O}$ be a domain. A $C^1(D)$-function $f:D \to \mathbb{O}$ that satisfies  $\frac{\partial f}{\partial x_0} \neq 0$ for all $z \in D$ is conformal in the sense of Gauss if there exist seven $C^0(D)$-functions $\Psi_i:D \to \mathbb{O}$ which obey 
\begin{eqnarray*}
	(i) & & \Re(\Psi_i(z)) = 0,\quad i=1,2,\ldots,7\\
	(ii) & & \langle \Psi_i(z),\Psi_j(z)\rangle = \delta_{ij} \quad \forall i,j\in \{1,2,\ldots,7\}
\end{eqnarray*}
such that 
\begin{equation}
\frac{\partial f}{\partial x_k} = \Psi_k(z) \frac{\partial f}{\partial x_0}.
\end{equation}
\end{theorem}
{\it Proof}. Suppose that $f$ is a $C^1(D)$-conformal map in $D$ with $\frac{\partial f}{\partial x_k} \neq 0$ for all $z \in D$. Exploiting the crucial property that the octonions form a division algebra one can well-define for each $k=1,\ldots,7$ at each  single point $z \in D$ the seven functions 
$$
\Psi_k(z) := \Big(\frac{\partial f}{\partial x_k}\Big) \Big(\frac{\partial f}{\partial x_0}\Big)^{-1}.
$$
Since $f \in C^1(D)$, the functions $\frac{\partial f}{\partial x_k}$ and $ \Big(\frac{\partial f}{\partial x_0}\Big)^{-1}$ belong at least to $C^0(D)$. Note that $\frac{\partial f}{\partial x_0} \neq 0$ for all $z \in D$. 

In view of $\Big|\frac{\partial f}{\partial x_0}\Big| = \Big|\frac{\partial f}{\partial x_j}\Big|$ for all $j=1,\ldots,7$ and applying the other crucial fact that $\mathbb{O}$ is a composition algebra where $|ab|=|a||b|$ for all $a,b \in \mathbb{O}$ we get that 
\begin{equation}\label{norm1}
|\Psi_k(z)| = \Bigg|\Big(\frac{\partial f}{\partial x_k}\Big) \Big(\frac{\partial f}{\partial x_0}\Big)^{-1}\Bigg| = \Big|\frac{\partial f}{\partial x_k}\Big|\cdot \Big|  (\frac{\partial f}{\partial x_0})^{-1} 
\Big| = \Big|\frac{\partial f}{\partial x_k}\Big| \cdot \frac{1}{\Big|\frac{\partial f}{\partial x_0}  \Big|} = 1,
\end{equation}
for all $z \in D$. 

Next we observe that 
\begin{eqnarray*}
	\Re(\Psi_k(z)) &=& \Re\Big\{(\frac{\partial f}{\partial x_k})(\frac{\partial f}{\partial x_0})^{-1}   \Big\} = \Bigg\{(\frac{\partial f}{\partial x_k}) \frac{\frac{\partial \overline{f}}{\partial x_0}}{\Big|\frac{\partial f}{\partial x_0}  \Big|^2}   \Bigg\} 
	 =  \frac{1}{\Big|\frac{\partial f}{\partial x_0}  \Big|^2} \Re \Big\{\frac{\partial f}{\partial x_k} \frac{\partial \overline{f}}{\partial x_0}  \Big\} = 0.
\end{eqnarray*}
Conversely, if $|\Psi_k(z)|=1$ for all $ z \in D$ then we have $|\frac{\partial f}{\partial x_0}| = |\frac{\partial f}{\partial x_k}|$ for all $k=1,\ldots,7$. 
\par\medskip\par

Now we recall that in the octonions we have the properties 
\begin{eqnarray}
\langle a,b \rangle &=& \Re\{a\overline{b}\} \label{scalar}\\
\Re\{b(\overline{a}a)c\} &=&\Re\{(b \overline{a})(ac)\} \label{moufang}
\end{eqnarray}
for all octonions $a,b,c \in \mathbb{O}$. In view of these rules we have for all $i,j \in \{1,\ldots,7\}$: 
\begin{eqnarray*}
\Re\Big\{ \frac{\partial f}{\partial x_i}\frac{\partial \overline{f}}{\partial x_j}   \Big\}	 &=& \Big\langle \frac{\partial f}{\partial x_i},\frac{\partial f}{\partial x_j}  \Big\rangle = \Big\langle \Psi_i(z) \frac{\partial f}{\partial x_0}, \Psi_j(z) \frac{\partial f}{\partial x_0}\Big\rangle \\
&=& \Re \Bigg\{\Big(\Psi_i(z) \frac{\partial f}{\partial x_0}\Big)\cdot \overline{\Big(\Psi_j(z) \frac{\partial f}{\partial x_0}\Big)}   \Bigg\}
= \Re \Bigg\{ \Big(\Psi_i(z) \frac{\partial f}{\partial x_0}\Big)\Big(\frac{\partial \overline{f}}{\partial x_0} \overline{\Psi_j(z)}\Big) \Bigg\}\\
& \stackrel{(\ref{moufang})}{=} & \Re \Bigg\{ \Psi_i(z) \cdot \Bigg[\Big(\frac{\partial f}{\partial x_0}  \Big)\Big(\frac{\partial \overline{f}}{\partial x_0}  \Big)  \Bigg]\cdot \overline{\Psi_j(z)}\Bigg\} \\
&=& \Big|\frac{\partial f}{\partial x_0} \Big|^2 \Re \{\Psi_i(z) \overline{\Psi_j(z)}\}\\
&=& \Big|\frac{\partial f}{\partial x_0} \Big|^2 \langle \Psi_i(z),\Psi_j(z)\rangle.
\end{eqnarray*}
\begin{remark}
This proof cannot be extended to the context of higher dimensional Cayley-Dickson algebras beyond the octonions because we explicitly used the property of a composition algebra and the property (\ref{moufang}) that relies on the structure of an alternative algebra. These two properties are essential.
\end{remark}
In the case of a paravector-valued function in an associative Clifford algebra, one can prove a similar statement, but the related structural systems $[\Psi] = (\Psi_1(z),\ldots,\Psi_n(z))$ have a different algebraic structure. They are {\em not} paravector-valued but they are bi-products of two paravectors with no-scalar part having norm $1$. For each single $z \in D$ they turn out to be elements of the spin group Spin$(n+1)$ which is the group generated by products of an even number of paravectors from $\mathbb{R}^{n+1}$ from the Clifford algebra $Cl_n$,  i.e.:
$$
{\rm Spin}(n+1) :=\{ a:= \prod\limits_{i=1}^{2k} a_i\mid a_i \in \mathbb{R}^{n+1},\; |a_i|=1\}.
$$ 
In this setting one has 
\begin{theorem}
Let $D \subseteq \mathbb{R}^{n+1}$ be a domain. A $C^1(D)$-paravector-valued function $f:D \to \mathbb{R}^{n+1} \subset Cl_n$ satisfying $\frac{\partial f}{\partial x_0} \neq 0$ at all $z \in D$ is conformal in the sense of Gauss if and only if there exist $n$ continuous functions $\Psi_1,\ldots,\Psi_n: D \to \mathbb{R}^n \oplus \mathbb{R}^{n,n} \subset {\rm Spin}(n+1)$ with 
\begin{eqnarray*}
	(i)' & & {\rm Sc}(\Psi_i(z)) = 0,\quad i=1,2,\ldots,n\\
	(ii)' & & {\rm Sc} \{ \Psi_i(z) \overline{\Psi_j}(z)\} = \delta_{ij} \quad \forall i,j\in \{1,2,\ldots,n\}
\end{eqnarray*}
at each single point $z \in D$ 
such that
$$
\frac{\partial f}{\partial x_k} = \Psi_k(z) \frac{\partial f}{\partial x_0},\quad k=1,\ldots,n.
$$
\end{theorem}
\begin{remark}
	The condition $(ii)'$ can be interpreted as orthogonality relation in terms of the usual Clifford inner product from \cite{bds} defined by $\langle a,b\rangle := {\rm Sc} (a\overline{b})$ which of course descends down to the elements of the spin group. 
\end{remark}
\begin{proof}
The proof can be done basically along the same lines as in the octonionic case, putting special emphasis to the following features. 

In constrast to the octonionic setting, the Clifford algebra $Cl_n$ possesses zero divisors for $n\ge 3$. However, all non-zero paravectors from $\mathbb{R}^{n+1}$ are invertible in $Cl_n$ and their inverses are again paravectors. Since $f$ is assumed to be a paravector-valued function, also its partial derivatives are paravector-valued, and therefore $(\frac{\partial f}{\partial x_0})^{-1}$ does exist whenever    	 $\frac{\partial f}{\partial x_0} \neq 0$. So, for each $z$ with $\frac{\partial f}{\partial x_0} \neq 0$ one can well-define the $n$ structural functions $\Psi_k(z) := (\frac{\partial f}{\partial x_k}) (\frac{\partial f}{\partial x_0})^{-1}$ for $k=1,\ldots,n$. Note that these indeed are composed by a product of two paravector valued expressions, since both functions $\frac{\partial f}{\partial x_k}$ and $(\frac{\partial f}{\partial x_0})^{-1}$ are paravector-valued. 

A further difference between the octonionic setting and the Clifford algebra setting consists of the fact that Clifford algebras are   no composition algebras anymore for $n \ge 3$. As mentioned in the previous section, the Clifford norm defined by $|a| = \sqrt{{\rm Sc}(a\overline{a})}$ is only a pseudo norm satisfying an inequality of the form $|ab| \le 2^{n/2} |a||b|$. However, if $a$ and $b$ are paravectors or more generally products of paravectors, then one still has that the composition property $|ab| = |a||b|$. Therefore, we may apply the same argument as in (\ref{norm1}) to conclude from $|\frac{\partial f}{\partial x_0}| = |\frac{\partial f}{\partial x_j}|$ for all $j=1,\ldots,n$ that $|\Psi_k(z)|=1$ in the Clifford norm for each $k=1,\dots,n$ and vice versa, too. 

Since we have the associativity property in the whole Clifford algebra, the second part of the proof stating that   
$$
{\rm Sc} \Big\{ \frac{\partial f}{\partial x_i} \frac{\partial \overline{f}}{\partial x_j}\Big\} = \Big| \frac{\partial f}{\partial x_0}\Big|^2 {\rm Sc} \{
\Psi_i(z) \overline{\Psi_j(z)}
\}, \quad\quad i,j \in \{1,\ldots,n\},\;\i\neq j
$$
can be carried over without any obstacle.  
\end{proof}	
Note that also this proof uses the composition property, so the restriction to paravector-valued functions is essential. 
 \begin{remark}
Note that none of the functions $\Psi_i$ has a scalar (real) part. In the octonionic case the orthonormality property implies that their anticommutator 
$\{\Psi_i,\Psi_j \}:=
\Psi_i(z) \Psi_j(z) + \Psi_j(z) \Psi_i(z)$ vanishes which means that $\Psi_i(z)	\Psi_j(z) = -\Psi_j(z) \Psi_i(z)$. This relation then is translated into the equaton: 
$$
\Bigg[ \Big(\frac{\partial f}{\partial x_i}  \Big)\Big( \frac{\partial f}{\partial x_0}   \Big)^{-1}  \Bigg] \cdot \Bigg[ \Big(\frac{\partial f}{\partial x_j}  \Big)\Big( \frac{\partial f}{\partial x_0}   \Big)^{-1}  \Bigg] = - \Bigg[ \Big(\frac{\partial f}{\partial x_j}  \Big)\Big( \frac{\partial f}{\partial x_0}   \Big)^{-1}  \Bigg] \cdot \Bigg[ \Big(\frac{\partial f}{\partial x_i}  \Big)\Big( \frac{\partial f}{\partial x_0}   \Big)^{-1}  \Bigg].
$$
In the Clifford algebra setting the situation is essentially different. 
For each $z \in D$ the functions $\Psi_i$ are elements of the spin group Spin$(n+1)$ being also mutually orthonormal to each other with respect to the Clifford scalar product. However, only in the special case where they are pure vectors, their anticommutator 
$\{\Psi_i,\Psi_j \}$ vanishes, while in the special case where the functions $\Psi_i$  are pure bivectors, their commutator $[\Psi_i,\Psi_j] := \Psi_i(z) \Psi_j(z) - \Psi_j(z) \Psi_i(z) $ vanishes. 
 Only in these two situations one obtains the special differential relations:
$$
\Bigg[ \Big(\frac{\partial f}{\partial x_i}  \Big)\Big( \frac{\partial f}{\partial x_0}   \Big)^{-1}  \Bigg] \cdot \Bigg[ \Big(\frac{\partial f}{\partial x_j}  \Big)\Big( \frac{\partial f}{\partial x_0}   \Big)^{-1}  \Bigg] = \pm \Bigg[ \Big(\frac{\partial f}{\partial x_j}  \Big)\Big( \frac{\partial f}{\partial x_0}   \Big)^{-1}  \Bigg] \cdot \Bigg[ \Big(\frac{\partial f}{\partial x_i}  \Big)\Big( \frac{\partial f}{\partial x_0}   \Big)^{-1}  \Bigg].
$$  
where we have the minus sign in the case of pure vectors and the plus sign in the case of pure bivectors.

However, in the {\em associative} cases working in a Clifford algebra one may directly simplify this equation into 
\begin{eqnarray*}
\Bigg[ \Big(\frac{\partial f}{\partial x_i}  \Big)\Big( \frac{\partial f}{\partial x_0}   \Big)^{-1}  \Bigg] \cdot \frac{\partial f}{\partial x_j} &=& \pm \Bigg[ \Big(\frac{\partial f}{\partial x_j}  \Big)\Big( \frac{\partial f}{\partial x_0}   \Big)^{-1}  \Bigg]  \cdot \frac{\partial f}{\partial x_i}\\
\Psi_i(z) \cdot \frac{\partial f}{\partial x_j} &=& \pm \Psi_j(z) \frac{\partial f}{\partial x_i}.
	\end{eqnarray*}
Multiplying both sides of the equation from the left with $\overline{\Psi_i(z)}$ leads to 
$$
(\overline{\Psi_i(z)} \Psi_i(z)) \frac{\partial f}{\partial x_j} = \pm \overline{\Psi_i(z)} \Psi_j(z) \frac{\partial f}{\partial x_i}.
$$
In view of the fact that $Sc(\Psi_i(z))  =0$ one has $\Psi_i(z) \overline{\Psi_i(z)}=1$  
Therefore, one may obtain the nice relation
$$
\frac{\partial f}{\partial x_j} = \pm \overline{\Psi_i(z)}\Psi_j(z) \frac{\partial f}{\partial x_i}
$$ 
which further gives
$$
\frac{\partial f}{\partial x_j} = \pm \overline{\Psi_i(z)}\Psi_j(z)\Psi_i(z) \frac{\partial f}{\partial x_0}.
$$
So, in general one has 
$$
\frac{\partial f}{\partial x_j} = (\pm 1)^{l/2} \overline{\Psi^{*}} \Psi_j \Psi^{*} \frac{\partial f}{\partial x_0},
$$
where $\Psi^{*}$ may be any finite product consisting of $2l$ factors built with the basis functions $\Psi_k$ for any $k=1,\ldots,n$. 
 \end{remark}

{\bf Further properties}:\\[0.2cm]
Both in the associative paravector-valued case and in the non-associative octonionic case we can relate the differential of a conformal map $f$ with the differential $dz^{[\Psi]} = dx_0 + \Psi_1 dx_1 + \cdots + \Psi_n dx_n$ where $[\Psi]:=[1,\Psi_1,\ldots,\Psi_n]$ is a local orthonormal basis system that may continuously vary from point to pint. 

More precisely, we get {\em locally} (as well as for the associative paravector-valued case (general $n$) and for the octonionic case $(n=7)$): 
\begin{eqnarray*}
df &=& \frac{\partial f}{\partial x_0} dx_0 +  \frac{\partial f}{\partial x_1} dx_1 + \cdots +  \frac{\partial f}{\partial x_n} dx_n\\
&=& \frac{\partial f}{\partial x_0} dx_0 + \Psi_1 \frac{\partial f}{\partial x_0} dx_1 + \cdots + 	\Psi_n \frac{\partial f}{\partial x_0} dx_n.
\end{eqnarray*}
Thus, we have locally 
$$
df = (dx_0 + \Psi_1 dx_1 + \cdots + \Psi_n dx_n) \frac{\partial f}{\partial x_0},\quad \langle \Psi_i,\Psi_j\rangle = \delta_{ij}
$$
which can be  translated into the form $df = dz^{[\Psi]} \cdot \frac{\partial f}{\partial x_0}$. This relation can formally be interpreted in terms of a generalized differential quotient 
$$
(dz^{[\Psi]})^{-1} (df) = \frac{\partial f}{\partial x_0}.
$$
Here, the linearization is understood with respect to the directions of the local orthonormal basis system $[\Psi]$ which may vary continuously at each point $z \in D$.   

\section{Discussion}

As mentioned in the introduction, the set of conformal maps coincides in the Euclidean spaces $\mathbb{R}^{n+1}$ with the set of the M\"obius transformations which is stated in Liouville's theorem. It can also be proved very elegantly with methods from Clifford algebras using the so-called cog-wheel lemma, cf. \cite{Cnops}. 

For convenience we recall from \cite{Burdik,DM,GTBook,MD} that the octonions offer an elegant representation of the $SO(9,1)$-M\"obius transformations (which can be mapped  bijectively to the set of octonionic conformal maps in the sense of Gauss). These can be written in the form $T(z)=(az+b)(cz+d)^{-1}$ with octonions $a,b,c,d \in \mathbb{O}$ satisfying additionally the so-called Manogue-Schrader conditions, see \cite{DM} for details.

In order to describe conformal maps from $\mathbb{R}^{n+1}$ to $\mathbb{R}^{n+1}$ in the associative Clifford algebra setting, one can desribe them in terms of the ($2\times2$) Clifford Ahlfors-Vahlen matrices. From \cite{Cnops,EGM87} and elsewhere one recalls that every conformal map $T: \mathbb{R}^{n+1} \to \mathbb{R}^{n+1}$ has a representation of the form $T(z) = (\alpha z +\beta)(\gamma z + \delta)^{-1}$ where $\alpha,\beta,\gamma,\delta$ are products of paravectors satisfying $\alpha \delta^{*} - \beta \gamma^{*} \neq 0$ and the Ahlfors-Vahlen conditions $\alpha^{-1}\beta, {\gamma^{-1}} \delta \in \mathbb{R}^{n+1}$ (if $\alpha,\gamma \neq 0$) which guarantee  that $T$ is really a map into $\mathbb{R}^{n+1}$. Such a map represents a transformation from $SO(n+1)$ if an only if $f(z)=\frac{\alpha}{|\alpha|} z \frac{\alpha^{*}}{|\alpha|}$.    

Applying the statements of Theorem 3.5 and Theorem 3.7 to the fact that all conformal maps have the form $f(z)=(az+b)(cz+d)^{-1}$ with certain octonionic coefficients or Clifford algebra valued coefficients, respectively, we can draw some conclusions on the nature of the related structural frame functions $\Psi_i(z)$. 

In the case where $[\Psi] = [e_1,e_2,\ldots,e_7]$ is the constant canonical set of the standard imaginary octonionic units, we have $(dz)^{-1} (df) = \frac{\partial f}{\partial x_0}$. This means, that the limit of the octonionic differential quotient $\lim\limits_{\Delta z \to 0} (\Delta z)^{-1} (\Delta f)$ exists. It equals the value of $\frac{\partial f}{\partial x_0}$ and we deal with right octonionic differentiable functions. However, the latter class is restricted to functions of the form $f(z)=za + b$ with octonions $a\neq 0$ and $b \in \mathbb{O}$. Conversely, a direct computation gives for octonionic functions of the form $f(z)=za+b$ that $\frac{\partial f}{\partial x_k} = e_k a$ and $\frac{\partial f}{\partial x_0} = a$. Thus, $\Psi_k(z) =( e_k a)\cdot a^{-1}$. Since the octonions form an alternative algebra, one particularly has $(e_k a) a^{-1} = e_k \cdot (a a^{-1}) = e_k$. Thus, we arrive at $\Psi_k(z) \equiv e_k$. 

In the case of working with paravector-valued M\"obius transformations in the Clifford algebra setting, a function $f(z)=z \alpha + b$ is only a M\"obius transformation if $\alpha$ is a non-zero real number. Otherwise, $z \alpha$ is not even a paravector. Only in this particular situation, the structural frame $[\Psi]$ equals the constant standard one $[e_1,e_2,\ldots,e_n]$.     

Furthermore, we observe that in the cases where we have a M\"obius transformation of the form $T(z)=(az+b)d^{-1} = (az)d^{-1} + bd^{-1}$ (referring both to the non-associative octonionic and paravector-valued associative case), the partial derivatives of $f$ with respect to the variables $x_0$ and $x_k$ are constants, because we deal with a linear function. In all these cases we get constant orthonormal frames $[\Psi_1,\ldots,\Psi_n]$. However, they are not equal to the standard orthonormal frame $[e_1,\ldots,e_n]$ because otherwise we would be in the case of right differentiable functions described above. 

Finally, as soon as one deals with M\"obius transformations of the form $f(z)=(az+b)(cz+d)^{-1}$ the partial derivatives are not constants anymore. Hence, the related structural functions $\Psi_k(z)$ are not constants, but continuous functions. 

\section{Acknowledgement} I wish to express my gratitude to Professor John Ryan from the University of Arkansas (Fayetteville, USA) for having suggested to me to also draw  my attention to this topic in the context of paravectors and for having suggested to me to consider structural sets involving bivector parts.  This input was very important for the development of this paper.

\end{document}